\documentclass[11pt]{article}
\usepackage[margin=1.1in]{geometry}
\usepackage{amsmath,amssymb,amsthm}
\usepackage{booktabs}
\usepackage[colorlinks=true,linkcolor=blue,citecolor=blue,urlcolor=blue]{hyperref}

\newtheorem{theorem}{Theorem}[section]
\newtheorem{lemma}[theorem]{Lemma}
\newtheorem{corollary}[theorem]{Corollary}

\theoremstyle{remark}
\newtheorem{remark}[theorem]{Remark}

\newcommand{\NA}{N_A}
\title{Exact $6$-cut rigidity and small-order superconnectivity\\
for the $6$-regular case of Dirac's $k=4$ problem}
\author{Alper Ferudun\thanks{\texttt{alper@mercurycodelab.com}. Code, certificates
and machine-checked Lean files accompany this note as ancillary files; development
history at \texttt{github.com/AlperTheKing}.}}
\date{June 2026}

\begin{document}
\maketitle

\begin{abstract}
Dirac asked in 1970 whether for every $k\ge 4$ there is a $k$-vertex-critical graph
without critical edges; Jensen settled all $k\ge 5$, and only $k=4$ remains open.
Following Skottova and Steiner, call a graph $G$ a \emph{$(4,1)$-graph} if
$\chi(G)=4$, $\chi(G-v)=3$ for every vertex $v$, and $\chi(G-e)=4$ for every edge
$e$; they proved $\delta(G)\ge 6$ and $\lambda(G)\ge 6$ for every $(4,1)$-graph and
asked whether a $6$-regular $(4,1)$-graph exists. We prove three results about this
$6$-regular case. \textbf{Theorem A} (computational): there is no
$6$-regular $4$-vertex-critical graph on $n\le 15$ vertices, except for a unique
graph (up to isomorphism) on $n=13$ vertices, whose $13$ critical edges form a
Hamilton cycle. Consequently any $6$-regular $(4,1)$-graph has at least $16$
vertices.
\textbf{Theorem B}: in a $6$-regular $(4,1)$-graph, every $6$-edge-cut is either the
edge star of a vertex or has both shores of size at least $15$; consequently every
$6$-regular $(4,1)$-graph on at most $29$ vertices is super-$6$-edge-connected (every
$6$-edge-cut is the edge star of a single vertex).
\textbf{Theorem C} (all sizes): no shore of a nontrivial $6$-edge-cut in a
$6$-regular $(4,1)$-graph induces a bipartite graph; more generally, a shore whose
deficiency is concentrated on two vertices forces them to receive equal colours in
every proper $3$-colouring. The proof plays the cut-matrix classification against
vertex-criticality and is sharp: the incidence graph of $\mathrm{PG}(2,5)$ yields
deficiency-$6$ candidates that pass all one-sided counting filters.
The proof of Theorem~B combines an exact classification of the $3\times 3$ cut
matrices of $6$-edge-cuts in $(4,1)$-graphs (exactly $21$ matrices, five types up to
row/column permutations) with four local necessary conditions on a hypothetical
shore, the strongest being a \emph{boundary-shortfall lemma} derived from
vertex-deleted colourings; a finite, exhaustively enumerated candidate space then
collapses. The unique near-miss is $K_{3,3,3}$ minus a rainbow $3$-matching (one edge
between each pair of parts, six distinct endpoints), which survives all counting
filters and is excluded by the boundary-shortfall lemma alone. Several supporting
lemmas are machine-checked in Lean~4/Mathlib (no \texttt{native\_decide}); the
shore-exclusion counts are independently reproduced in a separate Python
implementation for cut sizes $a\le 11$.
\end{abstract}

\section{Introduction}

A graph $G$ is \emph{$k$-vertex-critical} if $\chi(G)=k$ and $\chi(G-v)=k-1$ for
every vertex $v$. An edge set $R\subseteq E(G)$ is \emph{critical} if
$\chi(G-R)<\chi(G)$. Dirac (see Erd\H{o}s \cite{Er89}) asked in 1970 whether for
every $k\ge 4$ there exists a $k$-vertex-critical graph with no critical edge.
Jensen \cite{Je02} resolved all cases $k\ge 5$; the case $k=4$ remains open, the
only prior partial result being that of Jensen and Siggers \cite{JS12}, who
constructed $4$-vertex-critical graphs in which fewer than an $\varepsilon$-fraction
of the edges are critical. The problem is \#944 in Bloom's database of Erd\H{o}s
problems \cite{Bl944}.

Following Skottova and Steiner \cite{SkSt25} we call a graph $G$ a
\emph{$(4,1)$-graph} (for short, a \emph{target}) if
\[
\chi(G)=4,\qquad \chi(G-v)=3\ \ \forall v\in V(G),\qquad \chi(G-e)=4\ \
\forall e\in E(G).
\]
Skottova and Steiner proved that every $(4,1)$-graph satisfies $\delta(G)\ge 6$
and $\lambda(G)\ge 6$ (their Proposition~5.1) and posed as an explicit subproblem
(their Problem~5.2) whether a \emph{$6$-regular} $(4,1)$-graph exists --- the
extremal, sparsest possible case. This note contributes three verified results on
that subproblem: Theorems~A and~B below, and Theorem~C (a structural exclusion
valid at every size) in Section~\ref{sec:bipartite}.

\begin{theorem}[A]\label{thm:A}
There is no $6$-regular $4$-vertex-critical graph on $n\le 15$ vertices, except
for a unique graph $G_{13}$ (up to isomorphism) on $n=13$ vertices, given in
graph6 form in Section~\ref{sec:computations}; its $13$ critical edges form a
Hamilton cycle. Consequently every $6$-regular $(4,1)$-graph has at least $16$
vertices.
\end{theorem}

A graph is \emph{super-$6$-edge-connected} if it is $6$-edge-connected and every
$6$-edge-cut is \emph{trivial}, that is, the set of edges incident to a single
vertex (the edge star of that vertex).

\begin{theorem}[B]\label{thm:B}
In every $6$-regular $(4,1)$-graph, every $6$-edge-cut is either the edge star of a
single vertex or has both shores of size at least $15$. Consequently every
$6$-regular $(4,1)$-graph on at most $29$ vertices is super-$6$-edge-connected.
\end{theorem}

Theorem~B is obtained by an \emph{exact finite obstruction method}: the equality
case of the Skottova--Steiner cut bound forces every $6$-edge-cut of a target to
carry one of exactly $21$ cut matrices (five types up to row and column
permutations, Theorem~\ref{thm:cutmatrix}); this, together with three further local
necessary conditions --- the strongest being the boundary-shortfall lemma
(Lemma~\ref{lem:shortfall}) --- confines a hypothetical small shore to a finite,
explicitly enumerable family which then collapses entirely. We consider the method
itself, and the structural near-miss it uncovers ($K_{3,3,3}$ minus a rainbow
$3$-matching, Section~\ref{sec:nine}), to be of independent interest.

We are careful to separate known from new. The bounds $\delta\ge 6$, $\lambda\ge 6$
and the cut-averaging argument are due to Skottova and Steiner \cite{SkSt25};
the exclusion of comparable nonadjacent vertices in vertex-critical graphs is
folklore. The equality-case classification (Theorem~\ref{thm:cutmatrix}), the
boundary-shortfall lemma, the shore-exclusion theorems, and the computational
closure of $n\le 15$ appear to be new. None of this comes close to resolving
Problem~5.2: a target could still be a super-$6$-edge-connected $6$-regular graph
on $n\ge 15$ vertices, and our results do not constrain that core case.

Our computations are guarded as follows (see Section~\ref{sec:computations}). Every
enumeration uses \texttt{geng} with a generator-side completeness check; the
Theorem~A counts match the known numbers of $6$-regular graphs, and the prefiltered
and unfiltered classifiers agree on all of $n\le 13$. The Theorem~B shore-exclusion
battery is implemented independently in C\texttt{++} and Python with identical
counts for cut sizes $a\le 11$; the larger cases rely on a single C\texttt{++}
classifier with the generator-side completeness guard. Several supporting lemmas are
machine-checked in Lean~4/Mathlib. Full code, certificates, and Lean files are
included as ancillary files (we recommend these over the development repository,
whose contents may change).

\section{Local colouring lemmas}\label{sec:local}

Throughout, colourings are proper maps to $\{1,2,3\}$ (not counted up to colour
permutation). For $v\in V(G)$ and a $3$-colouring $\varphi$ of $G-v$ write
$a_i(v,\varphi)=|\{u\in N(v):\varphi(u)=i\}|$.

\begin{lemma}[singleton lemma]\label{lem:singleton}
Let $G$ be $4$-vertex-critical, $v\in V(G)$, and $\varphi$ a $3$-colouring of
$G-v$. Then $a_i(v,\varphi)\ge 1$ for every colour $i$; and if $a_i(v,\varphi)=1$,
say with unique colour-$i$ neighbour $u$, then the edge $vu$ is critical:
$\chi(G-vu)=3$.
\end{lemma}

\begin{proof}
If $a_i=0$, extend $\varphi$ by $\varphi(v)=i$; this $3$-colours $G$,
contradicting $\chi(G)=4$. If $a_i=1$ with witness $u$, keep $\varphi$ on
$G-v$ and set $\varphi(v)=i$ in $G-vu$: the only monochromatic edge would be
$vu$, which has been deleted.
\end{proof}

\begin{corollary}\label{cor:delta6}
In a target, $a_i(v,\varphi)\ge 2$ for all $v,\varphi,i$; hence $\delta(G)\ge 6$,
and in a $6$-regular target every neighbourhood splits exactly $2+2+2$ under every
vertex-deleted colouring. \textup{(}This is the local form of
\cite[Prop.~5.1]{SkSt25}.\textup{)}
\end{corollary}

\begin{corollary}\label{cor:quant}
Every $4$-vertex-critical graph $G$ has at least
$\tfrac12\sum_{v}\max(0,\,6-d(v))$ critical edges.
\end{corollary}

\begin{proof}
Fix for each $v$ one colouring $\varphi_v$ of $G-v$. The three positive integers
$a_i(v,\varphi_v)$ sum to $d(v)$, so at least $\max(0,6-d(v))$ of them equal $1$;
each singleton yields a critical edge at $v$ by Lemma~\ref{lem:singleton}, and
each critical edge is counted at most twice.
\end{proof}

\begin{lemma}[folklore]\label{lem:comparable}
A $k$-vertex-critical graph contains no two nonadjacent vertices $u,v$ with
$N(u)\subseteq N(v)$.
\end{lemma}

\begin{proof}
Take a $(k-1)$-colouring $\varphi$ of $G-u$ and set $\varphi(u)=\varphi(v)$:
every neighbour of $u$ is a neighbour of $v$, so no conflict arises and
$\chi(G)\le k-1$.
\end{proof}

We will also use the following observation, which gives a very effective
computational prefilter.

\begin{lemma}\label{lem:nbhd}
Let $G$ be a $4$-vertex-critical graph on $n\ge 8$ vertices with maximum degree at
most $6$. Then for every vertex $v$ the induced neighbourhood $G[N(v)]$ is bipartite.
\end{lemma}

\begin{proof}
Suppose $G[N(v)]$ contains an odd cycle; then it contains a chordless (induced) odd
cycle $C$. Since $|N(v)|=d(v)\le 6$, the length of $C$ is $3$ or $5$, so the wheel
$W=\{v\}\cup C$ is an induced subgraph on at most $6$ vertices, and it is
$4$-chromatic. As $n\ge 8>|W|$, there is a vertex $u\notin W$; then $W\subseteq G-u$,
so $\chi(G-u)\ge 4$, contradicting $4$-vertex-criticality.
The degree hypothesis is essential: without it an odd wheel (which is itself
$4$-vertex-critical) has a non-bipartite neighbourhood at its hub. In this note the
lemma is applied only to $6$-regular graphs, where $\Delta(G)=6$.
\end{proof}

\section{Exact rigidity of $6$-edge-cuts}\label{sec:cuts}

Let $G$ be a target and $V(G)=A\sqcup B$ a nontrivial partition with cut
$F=E(A,B)$. Both $G[A]$ and $G[B]$ are $3$-colourable (each is contained in some
$G-x$). For $3$-colourings $\alpha$ of $G[A]$ and $\beta$ of $G[B]$ define the
\emph{cut matrix} $M=(m_{ij})$ by
$m_{ij}=|\{ab\in F:\alpha(a)=i,\ \beta(b)=j\}|$, and for $\pi\in S_3$ the diagonal
sum $D_\pi(M)=\sum_i m_{i\pi(i)}$, the number of monochromatic cut edges after
permuting the colours of $B$ by $\pi$.

\begin{lemma}\label{lem:Dpi}
In a target: $D_\pi(M)=0$ is impossible; if $D_\pi(M)=1$ then the unique
monochromatic cut edge is critical. Hence $D_\pi(M)\ge 2$ for all $\pi$, and since
$\tfrac16\sum_{\pi}D_\pi(M)=|F|/3$, every cut satisfies $|F|\ge 6$
\textup{(\cite[Prop.~5.1]{SkSt25})}, and for $|F|=6$,
\[
D_\pi(M)=2\qquad\text{for all }\pi\in S_3\text{ and all pairs }(\alpha,\beta).
\]
\end{lemma}

\begin{proof}
$D_\pi=0$ gives a $3$-colouring of $G$; $D_\pi=1$ gives a $3$-colouring of $G-e$
for the unique monochromatic cut edge $e$. Each cut edge lies on exactly two of
the six permutation diagonals, giving the average $|F|/3$; for $|F|=6$ the average
equals the minimum, forcing equality everywhere.
\end{proof}

\begin{theorem}[$21$-matrix classification]\label{thm:cutmatrix}
Let $M$ be a $3\times 3$ matrix with nonnegative integer entries, total sum $6$,
and $D_\pi(M)=2$ for all $\pi\in S_3$. Then
\[
m_{ij}=\frac{R_i+C_j-2}{3}
\]
where $R_i,C_j$ are the row and column sums; all $R_i$ are congruent
$\bmod\ 3$, likewise all $C_j$; and $M$ is, up to row and column permutations, one
of the five types
\[
\begin{pmatrix}2&2&2\\0&0&0\\0&0&0\end{pmatrix},\
\begin{pmatrix}1&1&1\\1&1&1\\0&0&0\end{pmatrix},\
\begin{pmatrix}0&0&1\\0&0&1\\1&1&2\end{pmatrix},\
\begin{pmatrix}2&0&0\\2&0&0\\2&0&0\end{pmatrix},\
\begin{pmatrix}1&1&0\\1&1&0\\1&1&0\end{pmatrix},
\]
with orbit sizes $3,3,9,3,3$: exactly $21$ matrices in all. In particular the
row-sum vector is a permutation of $(6,0,0)$, $(3,3,0)$, $(4,1,1)$ or $(2,2,2)$.
\end{theorem}

\begin{proof}
Comparing two permutations that differ by a transposition and cancelling the
common entry gives $m_{rc}+m_{sd}=m_{rd}+m_{sc}$ for every $2\times 2$ submatrix;
hence $m_{pq}=m_{pj}+m_{iq}-m_{ij}$ for all $p,q,i,j$, and summing over all
$(p,q)$ yields $6=3R_i+3C_j-9m_{ij}$, i.e.\ the displayed formula. Integrality
forces $R_i+C_j\equiv 2 \pmod 3$ for all $i,j$, whence the congruence claims; the
nonnegative solutions of $\sum R_i=6$ with all $R_i$ congruent $\bmod\ 3$ are the
four listed vectors, and the case analysis over $(R,C)$ produces exactly the five
displayed types. The count $21$ is also machine-checked (Lean, by exhaustive
enumeration of the $3003$ weak compositions of $6$ into $9$ parts together with a
completeness lemma for the enumeration; see Section~\ref{sec:computations}).
\end{proof}

Since the row sums of $M(\alpha,\beta)$ do not depend on $\beta$, we obtain:

\begin{corollary}\label{cor:rowsum}
For every $3$-colouring $\alpha$ of a $6$-cut shore $G[A]$ in a target, the vector
$\bigl(\sum_{\alpha(v)=i} b(v)\bigr)_{i=1,2,3}$, where $b(v)$ is the number of cut
edges at $v$, is a permutation of $(6,0,0)$, $(3,3,0)$, $(4,1,1)$ or $(2,2,2)$.
\end{corollary}

\section{Shore constraints}\label{sec:shores}

From now on $G$ is a $6$-regular target and $A$ a shore of a nontrivial
$6$-edge-cut, $a=|A|\ge 2$, $H=G[A]$, $b(v)=6-d_H(v)$.

\begin{lemma}\label{lem:basic}
$e(H)=3a-3$; $H$ is connected; $0\le b(v)\le 5$ and $\sum_{v\in A}b(v)=6$.
\end{lemma}

\begin{proof}
Degree sum: $6a=2e(H)+6$. If $H$ had components $A_1,\dots,A_t$, the cuts
$\partial_G(A_i)$ are disjoint subsets of $\partial_G(A)$, each of size at least
$6$ by $\lambda(G)\ge 6$ \cite{SkSt25}, and their sizes sum to $6$; so $t=1$.
A vertex with $b(v)=6$ would be isolated in $H$.
\end{proof}

\begin{lemma}[boundary-shortfall lemma]\label{lem:shortfall}
For every $v\in A$ there exists a $3$-colouring $\psi$ of $H-v$ with
\[
\sum_{i=1}^{3}\max\bigl(0,\,2-c_i(\psi)\bigr)\ \le\ b(v),
\qquad\text{where } c_i(\psi)=|\{u\in \NA(v):\psi(u)=i\}|.
\]
\end{lemma}

\begin{proof}
Take a $3$-colouring $\varphi$ of $G-v$ and let $\psi=\varphi|_{A-v}$. By
Corollary~\ref{cor:delta6}, every colour appears at least twice on $N_G(v)$,
which is the disjoint union of $\NA(v)$ and the set $C_v$ of outside
cut-neighbours of $v$, $|C_v|=b(v)$. Writing $o_i$ for the number of vertices of
$C_v$ coloured $i$ we get $c_i+o_i\ge 2$, hence
$\max(0,2-c_i)\le o_i$, and summing over $i$ gives the bound since
$\sum_i o_i=b(v)$.
\end{proof}

For $v$ with $b(v)=0$ this says: \emph{$H-v$ must admit a colouring in which
$N(v)$ splits exactly $2+2+2$} --- full-degree vertices must be
``deletion-unfrozen''. Note also that for $u,v\in A$ nonadjacent with $b(u)=0$,
Lemma~\ref{lem:comparable} localizes: $\NA(u)\subseteq \NA(v)$ is impossible,
because then $N_G(u)=\NA(u)\subseteq N_G(v)$.

\section{Small shores}\label{sec:nine}

\begin{theorem}\label{thm:shore8}
No $6$-edge-cut in a target \textup{(}$6$-regular or not\textup{)} has a shore of
size $2\le a\le 8$.
\end{theorem}

\begin{proof}
By $\delta(G)\ge 6$, $e(H)\ge 3a-3$, while $H$ is $3$-colourable, so
$e(H)\le\lfloor a^2/3\rfloor$. For $2\le a\le 7$ these conflict
($\lfloor a^2/3\rfloor<3a-3$; this numeric fact is also Lean-checked). For $a=8$
both bounds equal $21$, forcing equality: $H$ is the Tur\'an graph
$K_{3,3,2}$ and every vertex of $A$ has $G$-degree exactly $6$. The two vertices
of the part of size $2$ then have all six edges inside $A$, so they are
nonadjacent vertices with equal neighbourhoods in $G$, contradicting
Lemma~\ref{lem:comparable}.
\end{proof}

\begin{theorem}\label{thm:shore913}
In a $6$-regular target, no $6$-edge-cut has a shore of size $9\le a\le 14$.
\end{theorem}

\begin{proof}[Proof (computer-assisted)]
By Lemma~\ref{lem:basic} a shore of size $a$ induces a connected graph on $a$
vertices with exactly $3a-3$ edges and maximum degree at most $6$. We enumerated
all such graphs with \texttt{geng} (nauty~2.8.9) and applied the necessary
conditions of Sections~\ref{sec:cuts}--\ref{sec:shores}: $3$-colourability;
Corollary~\ref{cor:rowsum} for \emph{every} $3$-colouring; the localized
Lemma~\ref{lem:comparable}; and Lemma~\ref{lem:shortfall} for every vertex. The
outcome:

\begin{center}
\begin{tabular}{rrrrrrr}
\toprule
$a$ & generated & not $3$-col. & fail Cor.~\ref{cor:rowsum} & fail
Lem.~\ref{lem:comparable} & fail Lem.~\ref{lem:shortfall} & survive\\
\midrule
9  & 729           & 711           & 9          & 8       & 1         & 0\\
10 & 18\,655        & 18\,345        & 197        & 86      & 27        & 0\\
11 & 696\,208       & 687\,377       & 6\,013      & 1\,300   & 1\,518     & 0\\
12 & 32\,833\,744    & 32\,484\,081    & 241\,863    & 27\,322  & 80\,478    & 0\\
13 & 1\,839\,349\,287 & 1\,822\,133\,664 & 11\,944\,366 & 788\,481 & 4\,482\,776 & 0\\
14 & 119\,236\,283\,370 & 118\,201\,012\,144 & 719\,933\,200 & 28\,146\,103 & 287\,191\,923 & 0\\
\bottomrule
\end{tabular}
\end{center}

No graph survives, proving the theorem. The verification protocol (generator-side
completeness guards, an independent Python reimplementation for $a\le 11$, and
per-graph certificates) is described in Section~\ref{sec:computations}.
\end{proof}

\begin{remark}[the near-miss at $a=9$]
Exactly one graph survives all the \emph{counting} filters at $a=9$:
$H_9=K_{3,3,3}\setminus M_3$, where $M_3$ is a rainbow $3$-matching --- one
missing edge between each pair of colour classes, with six distinct endpoints
(graph6 \texttt{HEzftz\{}). $H_9$ has a unique $3$-colouring up to
colour permutation; its three full-degree vertices --- one in each class --- have
all neighbours inside $A$. For each such vertex $v$, all six colourings of
$H_9-v$ leave some colour with at most one occurrence on $N(v)$, so
Lemma~\ref{lem:shortfall} (with $b(v)=0$) excludes $H_9$. This is the only point
at $a\le 13$ where the boundary-shortfall lemma is essential for a graph that the
other filters miss, and we verified the kill additionally by brute force over all
$3^{8}$ assignments.
\end{remark}

\begin{proof}[Proof of Theorem~\ref{thm:B}]
A nontrivial $6$-edge-cut has shores of sizes $\ge 2$; Theorems~\ref{thm:shore8}
and \ref{thm:shore913} exclude sizes $2..14$, so both shores have size at least
$15$, forcing $n\ge 30$. For $n\le 29$ every $6$-edge-cut is therefore trivial,
i.e.\ a vertex star.
\end{proof}

\section{Concentrated boundaries: bipartite shores are impossible}
\label{sec:bipartite}

The filters of Section~\ref{sec:shores} are one-sided: they constrain a shore in
isolation. We now add a genuinely two-sided constraint, obtained by playing the
cut matrices of Section~\ref{sec:cuts} against vertex-criticality, and use it to
exclude an entire structural class of shores at \emph{every} size.

Throughout, $A$ is a shore of a nontrivial $6$-edge-cut in a $6$-regular target
$G$, $H=G[A]$, and $B=V(G)\setminus A$ is the partner shore. Suppose the
deficiency of $A$ is \emph{concentrated}: $b(p)=b(\ell)=3$ for two vertices
$p\ne\ell$ and $b(v)=0$ otherwise. Write $T_p$ and $T_\ell$ for the sets of
endpoints in $B$ of the three cut edges at $p$ and at $\ell$ respectively. Call a
pair $(c_p,c_\ell)\in\{1,2,3\}^2$ \emph{realizable} if some proper $3$-colouring
$\alpha$ of $H$ has $\alpha(p)=c_p$ and $\alpha(\ell)=c_\ell$, and call $H$
\emph{pair-rich} if all nine pairs are realizable.

\begin{lemma}[pair rigidity]\label{lem:pairrigid}
A concentrated shore of a target is never pair-rich.
\end{lemma}

\begin{proof}
Suppose $H$ is pair-rich. First, every proper $3$-colouring $\beta$ of $G[B]$
makes $T_p$ or $T_\ell$ rainbow. Otherwise pick a colour
$c_p\notin\beta(T_p)$ and a colour $c_\ell\notin\beta(T_\ell)$; pair-richness
supplies $\alpha$ with $\alpha(p)=c_p$, $\alpha(\ell)=c_\ell$, and since the only
cut edges leave $p$ and $\ell$, the union $\alpha\cup\beta$ is a proper
$3$-colouring of $G$, contradicting $\chi(G)=4$.

Now take any $v\in A\setminus\{p,\ell\}$ (such $v$ exists: $a=2$ would force
$e(H)=3$ in a simple graph on two vertices). By vertex-criticality $G-v$ has a
proper $3$-colouring; restrict it to the intact $G[B]$ to get $\beta$. One of
$T_p,T_\ell$ is rainbow under $\beta$; the corresponding vertex ($p$ or $\ell$,
both still present in $G-v$) sees all three colours on its cut neighbours and on
its $H$-neighbours can avoid none, since all three colours already appear among
its three cut neighbours --- it cannot be coloured. Contradiction.
\end{proof}

\begin{lemma}[anti-diagonal exclusion]\label{lem:antidiag}
A concentrated shore of a target never satisfies
``$\alpha(p)\ne\alpha(\ell)$ for every proper $3$-colouring $\alpha$ of $H$''.
\end{lemma}

\begin{proof}
Suppose it does. Then for every $\alpha$ the row-sum vector of the cut matrix
$M(\alpha,\beta)$ is a permutation of $(3,3,0)$, and by
Theorem~\ref{thm:cutmatrix} the only matrices with these row sums are the
permutations of the second type, with both nonzero rows equal to $(1,1,1)$:
under \emph{every} proper $3$-colouring $\beta$ of $G[B]$, both $T_p$ and
$T_\ell$ are rainbow. As in Lemma~\ref{lem:pairrigid}, restricting a
$3$-colouring of $G-v$ (any $v\in A\setminus\{p,\ell\}$) to $G[B]$ now blocks
both $p$ and $\ell$ simultaneously --- contradiction.
\end{proof}

Since the set of realizable pairs is invariant under simultaneous colour
permutations, whose orbits on $\{1,2,3\}^2$ are the diagonal and the
off-diagonal, Lemmas~\ref{lem:pairrigid} and~\ref{lem:antidiag} leave exactly
one possibility.

\begin{corollary}[diagonal rigidity]\label{cor:diag}
In every concentrated shore of a target, $\alpha(p)=\alpha(\ell)$ for
\emph{every} proper $3$-colouring $\alpha$ of $H$; equivalently, $p\ell\notin
E(H)$ and $H+p\ell$ is not $3$-colourable.
\end{corollary}

\begin{theorem}[C]\label{thm:C}
In a $6$-regular $(4,1)$-graph, no shore of a nontrivial $6$-edge-cut induces a
bipartite graph.
\end{theorem}

\begin{proof}
Let $H$ be bipartite with parts $X\cup Y$. First, the deficiency of $H$ must be
concentrated. Indeed, colouring $X$ with colour $1$ and an arbitrary subset
$T\subseteq Y$ with colour $2$, $Y\setminus T$ with colour $3$, is always proper;
by Corollary~\ref{cor:rowsum} every subset sum $t$ of the $Y$-side deficiencies
must keep $\bigl(b(X),t,b(Y)-t\bigr)$ among the allowed vectors. Combined with
the plain bipartition colouring (which forces
$(b(X),b(Y))\in\{(6,0),(3,3),(0,6)\}$) and the same argument on the $X$ side,
the only solutions with all $b(v)\le 5$ are: a single vertex of deficiency $3$
on each side, or two vertices of deficiency $3$ on one side.

If $p,\ell$ are nonadjacent, $H$ is pair-rich: for $p,\ell$ in the same part
$X$, colour $p$ and $X\setminus\{\ell\}$ with $1$, $\ell$ with $2$, $Y$ with $3$
(off-diagonal pairs; the bipartition colouring gives diagonal pairs); for
$p\in X$, $\ell\in Y$ nonadjacent, colour $p,\ell$ with $1$, $X\setminus\{p\}$
with $2$, $Y\setminus\{\ell\}$ with $3$ (diagonal pairs; the bipartition
colouring gives off-diagonal pairs). Colour permutations complete all nine
pairs, contradicting Lemma~\ref{lem:pairrigid}.

If $p,\ell$ are adjacent (necessarily in different parts), every colouring has
$\alpha(p)\ne\alpha(\ell)$, contradicting Lemma~\ref{lem:antidiag}.
\end{proof}

\begin{remark}[a size barrier for one-sided filters]\label{rem:pg}
Theorem~\ref{thm:C} is not a luxury: the one-sided filters of
Section~\ref{sec:shores} genuinely fail to see large bipartite shores. The
incidence graph of the projective plane $\mathrm{PG}(2,5)$ ($31$ points, $31$
lines, $6$-regular bipartite, girth~$6$, $62$ vertices) has the property that
for \emph{every} vertex $v$ there is a proper $3$-colouring of the graph minus
$v$ whose colour counts on $N(v)$ are exactly $(2,2,2)$ --- verified by two
independent implementations, with point- and line-transitivity reducing the
check to two orbits. Consequently the boundary-shortfall lemma
(Lemma~\ref{lem:shortfall}) excludes \emph{no} vertex of suitable deficiency-$6$
modifications of this graph: one can delete three edges at a point $p$ and three
at a line $\ell$ and re-match the six unit deficits along non-incidences,
producing concentrated candidates that pass all the one-sided filters of
Section~\ref{sec:shores}. Theorem~\ref{thm:C} excludes them all the same. This
also shows that the analogue of the boundary-shortfall mechanism for whole
graphs (``every connected $6$-regular $3$-colourable graph has a vertex with no
balanced deletion colouring''), true exhaustively for $n\le 15$, fails at
$n=62$.
\end{remark}

\section{Computations, certificates, and machine-checked lemmas}
\label{sec:computations}

\paragraph{Theorem A.} All $6$-regular graphs on $7\le n\le 15$ vertices were
enumerated with \texttt{geng -d6 -D6} (nauty 2.8.9) and classified by a C\texttt{++}
backtracking $3$-colourability checker into: $3$-colourable; $\chi\ge 4$ but not
vertex-critical; $4$-vertex-critical with a critical edge; target. (A $6$-regular
graph requires $n\ge 7$, with $K_7$ the only one on $7$ vertices.) Counts:

\begin{center}
\begin{tabular}{rrrrrr}
\toprule
$n$ & total & $3$-colourable & not vertex-critical & $4$-VC with critical edge &
target\\
\midrule
7  & 1 & 0 & 1 & 0 & 0\\
8  & 1 & 0 & 1 & 0 & 0\\
9  & 4 & 1 & 3 & 0 & 0\\
10 & 21 & 1 & 20 & 0 & 0\\
11 & 266 & 3 & 263 & 0 & 0\\
12 & 7\,849 & 50 & 7\,799 & 0 & 0\\
13 & 367\,860 & 849 & 367\,010 & 1 & 0\\
14 & 21\,609\,301 & 42\,667 & 21\,566\,634 & 0 & 0\\
15 & 1\,470\,293\,676 & \multicolumn{2}{c}{(combined, see text)} & 0 & 0\\
\bottomrule
\end{tabular}
\end{center}

The $n=14$ run was executed twice with different residue-class chunkings
(mod~$110$ and mod~$73$) with identical aggregates, and the total matches the
known count of $6$-regular graphs on $14$ vertices. The $n=15$ run used the
neighbourhood-bipartiteness prefilter of Lemma~\ref{lem:nbhd} (applicable since the
graphs are $6$-regular, so $\Delta\le 6$), which eliminates $\approx 90\%$ of
candidates; the prefilter implementation was validated against the unfiltered
classifier on all of $n=11,12,13$ with identical results, and every residue class
was closed with a generator-side completeness guard (the generator's terminating
count must equal the classifier's count). The unique $n=13$ graph is
\begin{center}
\texttt{L?bFFbw\textasciitilde B\{FwFw}
\end{center}
Its $13$ critical edges form a Hamilton cycle. Across its $13$ vertex-deleted
subgraphs it admits exactly $78$ proper $3$-colourings (as maps to $\{1,2,3\}$,
not modulo $S_3$), and \emph{none} of them splits the open neighbourhood
$2+2+2$: every colouring leaves a singleton class in $N(v)$, so
Lemma~\ref{lem:singleton} predicts critical edges at every vertex; the $156$
singleton predictions hit exactly the $13$ critical edges (no false positives).
Thus the unique smallest $6$-regular $4$-vertex-critical graph fails the target
conditions everywhere locally, not marginally.

\paragraph{Theorem B.} Candidate shores were generated by
\texttt{geng -c -D6 $a$ $E$:$E$} with $E=3a-3$; large cases ran as residue
classes with a generator-side completeness guard (each class must end with
\texttt{geng}'s \texttt{>Z} terminator line, and the generated count must equal
the classifier's count). The filter battery was implemented twice independently
(C\texttt{++} and Python) with identical counts at $a=9,10,11$; the $a=9$ kill was
additionally brute-forced. Per-graph certificates (graph6 plus the failed filter,
and for shortfall kills the killing vertex) are included for $a\le 11$ in the
ancillary files.

\paragraph{Machine-checked lemmas.} The following are verified in Lean~4 /
Mathlib (no \texttt{sorry}, no \texttt{native\_decide}; axioms at most
\texttt{propext}, \texttt{Classical.choice}, \texttt{Quot.sound}):
the recolouring core of Lemma~\ref{lem:singleton} (axiom-free); the count of
$21$ in Theorem~\ref{thm:cutmatrix}, via an enumeration of the $3003$ weak
compositions of $6$ into $9$ parts together with a proved completeness lemma for
the enumeration, and a bridge lemma stating the matrix form with
$\forall\pi\in S_3$ diagonal bounds; the numeric inequality
$\lfloor a^2/3\rfloor<3a-3$ for $2\le a\le 7$ used in Theorem~\ref{thm:shore8};
and the row-sum forcing used in Theorem~\ref{thm:C} (among the $21$ valid cut
matrices, every one with row-sum vector $(3,3,0)$ has both nonzero rows $(1,1,1)$,
and every one with $(6,0,0)$ has nonzero row $(2,2,2)$). Each cited statement is
machine-checked; computations not on this list (the graph enumerations) are not
formalised.

\paragraph{Reproducibility.} The ancillary files are self-contained and do not
depend on any external repository. They include: the C\texttt{++} classifiers and
the shore-filter battery (with the independent Python reimplementation), the Lean~4
source of the machine-checked lemmas, per-graph certificates for $a\le 11$, the
graph6 string of $G_{13}$, the $\mathrm{PG}(2,5)$ verification scripts, a
\texttt{README} giving the exact \texttt{geng} (nauty~2.8.9) commands and residue
chunkings used for each row of each table, and a \texttt{SHA256SUMS} manifest of all
files. All generation uses nauty~2.8.9 \texttt{geng}; each enumeration is closed by
requiring \texttt{geng}'s terminating \texttt{>Z} count to equal the classifier's
count.

\section{Concluding remarks}

Theorem~B turns the Skottova--Steiner edge-connectivity threshold into an exact
small-cut impossibility statement, but it says nothing about
super-$6$-edge-connected candidates, which is where Problem~5.2 now lives: a
$6$-regular target must have $n\ge 15$ (Theorem~A), and if $n\le 29$ it has no
nontrivial $6$-edge-cut at all (Theorem~B). The shore-exclusion machine is
open-ended --- each further size $a$ is a finite computation; the $a=14$ case
($119{,}236{,}283{,}370$ candidates, every residue class closed with a
generator-side completeness guard) completed in under a day on a workstation --- but we
do not currently see an argument that it terminates for all $a$: the row-sum
filter can be neutralized by small colour-forcing gadgets, and the burden falls
entirely on the boundary-shortfall lemma. Proving that
Corollary~\ref{cor:rowsum} forces enough colour rigidity for
Lemma~\ref{lem:shortfall} to fail at some full-degree vertex, for every $a$,
would upgrade Theorem~B to: \emph{every $6$-regular $(4,1)$-graph is
super-$6$-edge-connected}. We leave this as the natural next question.

\end{document}